\documentclass{amsart}

\pdfoutput=1

\usepackage{amssymb}
\usepackage{amsmath, amsthm}

\usepackage[vlined, ruled, linesnumbered, norelsize]{algorithm2e}

\newtheorem{theorem}{Theorem}[section]
\newtheorem{lemma}[theorem]{Lemma}
\newtheorem{proposition}[theorem]{Proposition}

\theoremstyle{definition}
\newtheorem{definition}[theorem]{Definition}

 \newtheorem{heuristic}{Heuristic}

\theoremstyle{remark}

\numberwithin{equation}{section}

\newcommand{\Z}{ \mathbb{Z}}  

\newcommand{\Ot}{ \widetilde{O}}  
\newcommand{\Lunits}{ (\Z/\Lambda \Z)^{\times} } 
\newcommand{\La}{ \Lambda}

\newcommand{\E} { {\rm E}}

\begin{document}

\title[constructing Carmichael numbers]
{constructing Carmichael numbers through improved subset-product algorithms}

\author[Alford]{W.R. Alford } 

\author[Grantham]{Jon Grantham}
\address{Institute for Defense Analyses \\ Center for Computing Sciences \\
		17100 Science Drive \\ Bowie, MD 20715, United States}

\email{grantham@super.org}

\author[Hayman]{Steven Hayman}
\address{Illinois Wesleyan University \\	
		1312 Park St \\ Bloomington, IL 61701, United States}
\email{shayman@iwu.edu}

\author[Shallue]{Andrew Shallue}
\address{Illinois Wesleyan University \\
		1312 Park St \\ Bloomington, IL 61701, United States}
\email{ashallue@iwu.edu}
\thanks{Research supported by an Illinois Wesleyan University grant}
\thanks{W.R. Alford passed away in 2003}




\date{}


\begin{abstract}
We have constructed a Carmichael number with 
10,333,229,505 prime factors, 
and have also constructed Carmichael numbers with $k$ prime factors for every 
$k$ between 3 and 19,565,220.  These computations are the product 
of implementations of two new algorithms for the subset product problem 
that exploit the non-uniform distribution of primes $p$ with the property 
that $p-1$ divides a highly composite $\Lambda$.
\end{abstract}
\maketitle


\section{Introduction}

A Carmichael number $n$ is a composite integer that is a base-$a$ Fermat pseudoprime 
for all $a$ with ${\rm gcd}(a,n) = 1$.  However, constructions of Carmichael numbers often 
rely on the following equivalent definition.

\begin{definition}[Korselt condition] \label{def:korselt}
A positive integer $n$ is a Carmichael number if it is composite, squarefree, and has 
the property that $p-1 \mid n-1$ for all primes $p$ dividing $n$.
\end{definition}

Our goal is to construct Carmichael numbers with a very large number of prime factors.
The construction we will use is due to Erd\H{o}s \cite{Erdos56} and 
 has been a popular method since 1992 \cite{Zh92}.
 
\bigskip

{\bf Erd\H{o}s Construction:}
\begin{enumerate}
\item Choose  $\La = \prod_{i=1}^r q_i^{h_i}$ where 
$q_1 \dots q_r$ are the first $r$ primes in order and the $h_i$ are 
all at least $1$ and non-increasing.
\smallskip

\item Construct the set $\mathcal{P} = \{ \mbox{$p$ prime} \ : \ p-1 \mid \La, p \nmid \La \}$
\smallskip

\item 
Construct Carmichael $n$ as a product of primes in $\mathcal{P}$ in one of two ways:
\begin{enumerate}
\item \label{construction1}
Find a subset $\mathcal{S}$ of $\mathcal{P}$ such that 
$$
\prod_{p \in \mathcal{S}} p \equiv 1 \mod \La \enspace .
$$
Then by Definition \ref{def:korselt} $n = \prod_{p \in \mathcal{S}} p$ is Carmichael.
\smallskip

\item \label{construction2}
Alternatively, let $b \equiv \prod_{p \in \mathcal{P}} p \mod{\La}$ and 
 find a subset $\mathcal{T}$ of $\mathcal{P}$ such that 
$$
\prod_{p \in \mathcal{T}} p \equiv b \mod \La \enspace .
$$
Then $n = \prod_{p \in \mathcal{P} \setminus \mathcal{T}} p$ is Carmichael.
\end{enumerate}
\end{enumerate}

This notation for $\La$ will be fixed throughout, along with $b$ as the product modulo
$\La$ of all primes in $\mathcal{P}$.  Additionally, we will use $Q_i$
to represent $q_i^{h_i}$ and fix the ordering so that $q_1 = 2$, $q_2 = 3$ and so on.  

Choosing a good $\La$ is something of an art, seeing as how we want a number 
that is small and yet has many divisors.
One possibility (taken as a starting point by the authors in \cite{LohNie96}) is to 
choose $\La$ to be highly composite \cite{Ram97}.  We do not insist upon it here, instead 
relying on the condition that $h_i \leq r/i$ for $1 \leq i \leq r$ in order to prove bounds 
on running times. 
In practice, an excellent tool for choosing $\La$ is the following function $K(\La)$ 
from \cite{LohNie96} that returns an estimate of $|\mathcal{P}|$.
$$
K(\La) = \left \lfloor \frac{\Lambda}{\phi(\La) \log{\sqrt{2\La}}} 
	\prod_{i=1}^r \left( h_i + \frac{q_i - 2}{q_i - 1} \right) \right \rfloor 
$$
In terms of constructing $\mathcal{P}$, all divisors $d$ of $\La$ are checked to see if 
$n=d+1$ is prime.  Primality proofs are easy since we are given the factorization of $n-1$.
They are also necessary; the second author has found pseudoprimes while 
testing primality using randomized algorithms.

L\"{o}h and Niebuhr noted that step (2) is by far the most costly.  Nevertheless, 
the improvements we present will be to step (3), seeing as how the subset product 
problem in such a dense set of instances is fertile ground for algorithmic advancement.
In addition, step (2) is easily parallelized while step (3) is not, and step (2) requires 
almost no memory while the space requirement of most subset product algorithms is high.

Our new contribution involves improvements to existing literature on a broad array of fronts.
One new algorithm combines a series of smaller Carmichael numbers into larger ones, 
enabling quick construction of Carmichaels with a variety of factor counts.  Another new algorithm
incorporates ideas from  \cite{LohNie96} and \cite{FlaxPrz05} to achieve a randomized
method that solves the subset product problem using subexponential time and space
 (in fact, sublinear in $N$).  Computationally this has resulted in Carmichael numbers 
 with $10333229505$ prime factors and with $k$ prime factors for every $k$ between 
 $3$ and $19565220$.
 
 One key insight is that elements of $\mathcal{P}$ are not distributed uniformly 
 in $\Lunits$, and this non-uniformity can be exploited.  Another is that it is useful 
 to make products that get successfully closer to the identity in $\Lunits$.
 We will use two such functions for our algorithms, the first of which comes from 
 \cite{LohNie96}.

\begin{definition} \label{def:omega}
Let $a \in \Lunits$.  Then $\omega(a)$ is an integer between $0$ and $r$ defined by
$$
\omega(a) = \max_{a \bmod{Q_i} \neq 1} i
$$
unless $a \bmod{Q_i} = 1$ for all $1 \leq i \leq r$, in which case $\omega(a) = 0$.
\end{definition}

\begin{definition} \label{def:omegabar}
Let $a \in \Lunits$.  Then $\bar{\omega}(a)$ is an integer between $0$ and $r$ defined by
$$
\bar{\omega}(a) = \min_{a \bmod{Q_i} \neq 1} i
$$
unless $a \bmod{Q_i} = 1$ for all $1 \leq i \leq r$, in which case $\bar{\omega}(a) = r+1$.
\end{definition}

Formally, the subset product problem over an abelian group is 
defined as follows. 

\begin{definition}
Let $G$ be an abelian group written multiplicatively with  $(a_1, \dots, a_N, b)$ a list 
of elements of $G$.  Then the subset product problem $(a_1, \dots, a_N,b)$ is 
to determine a sublist of the $a_i$ that product to $b$ in $G$.
\end{definition}

In the Erd\H{o}s construction $G$ is typically $\Lunits$, but 
we will also consider subset product problems on subgroups of $\Lunits$.

The subset product problem is NP-hard, but the difficulty of a particular instance can vary 
depending on its density.  The hardest problems are those of density $1$.

\begin{definition}
The density of a subset product problem $(a_1, \dots, a_N,b)$ is 
$$
\frac{N}{\log_2{|G|}} \enspace .
$$
\end{definition}

Problems arising from the Erd\H{o}s construction will typically have density much 
larger than $1$, in fact closer to $O(N/\log{N})$.  We thus expect algorithms to exist 
with running times much faster than $O(2^N)$.
Since so many solutions are available, we are free to focus on a solution
with special properties that makes it 
easier to find.

Many algorithms for subset sum and subset product are randomized, and hence rely on 
an assumption about the distribution of the $a_i$.  A little experimentation reveals that 
elements of $\mathcal{P}$ are not uniformly distributed in $\Lunits$, but instead 
are close to a symmetric distribution (see Section \ref{sec:dist}).

\begin{definition}
A random variable $X$ on a group $G$ follows a symmetric distribution if for every 
$a \in G$, ${\rm Pr}[ X = a] = {\rm Pr}[ X = a^{-1}]$.
\end{definition}

In what follows we give a new algorithm for the random subset product problem that works
for instances of high density on groups of the form given in the above construction.
Specifically we prove the following two theorems.

\begin{theorem} \label{theorem:main1}
Let $G= G_0$ be an abelian group with subgroups $G_1, G_2, \dots, G_{\ell}$ where 
$|G_i|/|G_{i+1}| = 2^{\ell}$ for $0 \leq i \leq \ell-1$ and $\ell = \sqrt{\log{|G|}}$.
Assume that $a_1, \dots, a_N$ are independent and distributed symmetrically in $G$, 
and that $N > O( 4^{\sqrt{\log{|G|}}} \log{|G|})$.

Then there is an algorithm that solves the subset product problem 
$(a_1, \dots, a_N, b)$ with high probability and requires time and space
$$
\Ot \left( 4^{\sqrt{\log{|G|}}} \right) \enspace .
$$
\end{theorem}

\begin{theorem} \label{theorem:main2}
Let $G = \Lunits$ and $\mathcal{P}$ be defined as in the Erd\H{o}s construction, with the 
added assumption that $1 \leq h_i \leq r/i$ for all $1 \leq i \leq r$.
Assume that the elements of $\mathcal{P}$ are independent and 
distributed symmetrically in $G$, and 
that $N = | \mathcal{P} | = K(\La)$.  Finally, assume that the probability $p \equiv 1 
\bmod{Q_i}$ for $p \in \mathcal{P}$ is at least $1/(h_i+1)$ and independent across $Q_i$.

Then there is an algorithm that with high probability 
finds a subset of $\mathcal{P}$ that products to $b$ in $G$ and requires time and space
$$
2^{ O( \sqrt{  (\log{N}) (\log\log{N})^2 } ) } \enspace .
$$
\end{theorem}

The symbol $\log$ will denote the base $2$ logarithm, while $\ln$ denotes 
the natural logarithm.  Groups are assumed to have efficient implementations of 
arithmetic.

Thanks to Eric Bach and Carl Pomerance for helpful suggestions.
The third and fourth authors are grateful to Mark Liffiton for helpful advice and support 
regarding both hardware and software.  

\section{Previous Results}
Constructing Carmichael numbers has a long history, one that is ably documented 
in \cite{LohNie96} and \cite{Pinch93}.  We restrict ourselves to pointing out the more 
recent results that provide context for our new computations.

The largest tabulation of Carmichael numbers is due to Richard Pinch; his tabulation 
up to $10^{15}$ \cite{Pinch93} has since been extended to $10^{16}$.  Alford, Granville,
and Pomerance proved there are infinitely many Carmichael numbers in \cite{GranPom94}.
The authors credit their inspiration to \cite{Zh92}, who first used 
the Erd\H{o}s heuristic to construct Carmichael numbers with a large number of prime factors.
Current records for Carmichael numbers with many prime factors go 
to L\"{o}h and Niebuhr \cite{LohNie96} at $1101518$ prime factors and an unpublished 
computation by the first two authors at $19565300$ prime factors.

The method of \cite{LohNie96} clearly works well in practice, but unfortunately is 
without a running time analysis.  This makes it difficult to fit into the existing subset product 
literature since it is not clear how the running time depends on $N$ or on $\La$.
The algorithm exploits the fact that among elements of $\mathcal{P}$,
residues of $1$ modulo $q_i^{h_i}$ are more common than other residues.
It divides $b$ by $p \in \mathcal{P}$ in such a way that the running product has 
one more residue equal to one at each step, backtracking if necessary.  Measuring 
the closeness of an element of $\Lunits$ to the identity is done via Definition 
\ref{def:omega}. 

There is a large body of literature on the subset sum problem that transfers 
immediately to the subset product problem.
For subset product problems of high density the standard technique is dynamic programming, 
which in this case would take $O(N \La)$ time and space.  Since our goal is to construct 
Carmichael numbers where $N$ is in excess of $2^{30}$ and $\La$ is even bigger, this 
method is infeasible. 
A better naive algorithm is to pick 
a random subset and see if it products to $b \bmod \La$.  Even if the 
elements of $\mathcal{P}$
were distributed uniformly the expected time taken would be $\Ot( \phi(\La))$.  
 The polynomial time algorithm of \cite{Cha99} is similarly infeasible;
with $N$ so large we need an algorithm that is sub-linear in $N$.

Wagner's algorithm for the $k$-tree problem \cite{Wagner02}
 has inspired an algorithm for the subset sum problem that gets faster as the 
 density increases \cite{Lyu05, Sha08, MinAlis09}.  A recent paper by 
 Howgrave-Graham and Joux \cite{HGJoux10} even gives improvements for 
 most problems of density $1$.  However, all these methods are exponential time 
 and thus inappropriate for our setting.

Theorem \ref{theorem:main1} was inspired by the Kuperberg sieve from the theory 
of quantum algorithms \cite{Kup05}, which has complexity $2^{O(\sqrt{\log{|G|}})}$ where 
elements being matched are in a group $G$.  The same idea of combining 
elements in pairs to zero out square root of the bit size at each step was presented 
by Flaxman and Przydatek \cite{FlaxPrz05}, though they chose to only apply their 
algorithmic idea to a narrow slice of problems with the proper density to make the algorithm 
run in polynomial time.  A better application to the current setting would be 
to pick at least $2^{\sqrt{\log{\La}}}$ elements of $\mathcal{P}$ and pair them 
up over $\sqrt{\log{\La}}$ levels, zeroing out $\sqrt{\log{\La}}$ bits of $\La$ at each 
level.  The algorithm in Theorem \ref{theorem:main2} does even better by applying 
the Kuperberg idea to a subgroup of $\Lunits$, and the surprise is that there are enough 
elements of $\mathcal{P}$ that fall in the subgroup so that the algorithm can still succeed.


\section{Algorithms}\label{sec:algorithms}

Among the two new subset product algorithms presented in this section, Algorithm
\ref{alg:grantham} (developed by the first two authors) 
is more appropriate for constructing Carmichael numbers with 
$k$ prime factors for a variety of $k$, while Algorithm \ref{alg:shallue} (developed by 
the third and fourth authors) is better 
at constructing Carmichael numbers with a very large number of prime factors.
Both build products of primes in $\mathcal{P}$ that get successively close 
to the identity in $\Lunits$, with Algorithm \ref{alg:grantham}
utilizing Definition \ref{def:omegabar} while Algorithm \ref{alg:shallue} uses
Definition \ref{def:omega}.

The motivation behind Algorithm 1 was the observation by the first
author that his implementation of the Erd\H{o}s heuristic only needed to
use a small fraction of the available primes to generate Carmichael
numbers.  Through repeated runs, each time removing the
primes comprising the previous Carmichael number, it produces a set of
co-prime Carmichael numbers where the product of any subset also
forms a Carmichael number. Because of the inevitable difference in
numbers of prime factors, it is likely that the sums of the individual
numbers of prime factors will cover a wide range of possibilities.

In order to maximize the chance of getting most of the intermediate
numbers of prime factors, we want to generate as many different
Carmichael numbers as possible with relatively
few prime factors.  We achieve this goal over $r$ stages (one for each $Q_i$), 
where at stage $j$ we work with a set $S_j$ containing 
products $a$ with $\bar{\omega}(a) =j$ (we call this set $S$ for simplicity).
Let $S_1=\mathcal{P}$, and let $b_j = \prod_{a \in S_j} a \bmod{\La}$.  
For each element in
$S_j$, calculate its residue modulo each of the prime powers dividing
$\Lambda$. Do the same for each $b_j$.  

First, if $b_j \not \equiv 1 \bmod{Q_j}$ find and remove the element 
$e \in S$ that maximizes $\bar{\omega}(e^{-1}b_j)$.
As long as there is an
element congruent to the product modulo $Q_j$ in stage $j$ (which there
will almost certainly be in high density situations), the new product of
all primes in $S$ will be $1$ modulo $Q_i$ for all $i\le j$. 
Then for
each remaining element of $a\in S$, we use a greedy algorithm to find
$\hat{a} \in S$ maximizing $\bar{\omega}(a \cdot \hat{a})$. 
For simplicity Algorithm \ref{alg:grantham} shows $\bar{\omega}(a \cdot \hat{a})$
increasing by only one, but an important improvement is to efficiently find 
$\hat{a}$ that maximizes the increase to the $\bar{\omega}$ value 
(see Section \ref{sec:implementation} for details).
At the end of this process, you will have a (potentially
empty) set of elements that were not matched. At this point, you can
multiply all of these ``chaff'' together to get an element that is $1$
modulo $Q_j$, as well as being $1$ modulo $Q_i$ for $i < j$. 
Replace $S$ with the set-aside products, along with the
product of the chaff, and move on to the next stage.

At the end of the $r$ stages, you have a collection of coprime
``base'' Carmichael numbers. It is not necessary to compute Carmichael
numbers with a particular number of prime factors to prove its
existence. Instead, let $v$ be the vector $[1,0,0,\dots,0]$ of length
equal to one greater than the number of total prime factors in the
Carmichael numbers.  Loop over the base Carmichael numbers. For each
number, let $v'$ be $v$ shifted to the right by the number of factors
in that Carmichael number. E.g., if the first base Carmichael has $3$
factors, $v'=[0,0,0,1,\dots]$. Let $v=v+v'$. Then, at the end of the
loop, the $k$th position in the vector will represent the number of
constructed Carmicheals with $k-1$ prime factors. (Excepting the first
position.)

\begin{algorithm} \label{alg:grantham}
\caption{Many Carmichaels subset-product} 
\SetKwFunction{Sort}{sort}
$S_1 \leftarrow \mathcal{P}$ \;
\For{$u \leftarrow 1$ \KwTo $r$}
{  \Sort $S_u$\;
    calculate $b_u = \prod_{a \in S_u} a \bmod{Q_u}$, remove $a$ from $S_u$ that 
      satisfies $a \equiv b_u \bmod{Q_u}$ \;
   \For{ $a \in S_u$}
    { Find $\hat{a} \in S_u$ such that $a \cdot \hat{a} \equiv 1 \bmod{Q_u}$ \tcc*[r]{pushing down}
       $S_{u+1} \leftarrow a\hat{a}$ \;
        Remove $a,\hat{a}$ from $S$ \;}
      product remaining $a \in S_u$, add to $S_{u+1}$ \;
}
 \Return $S_{r+1}$, each element of which is Carmichael \;
\end{algorithm}

Our next algorithm constructs a Carmichael number via Step \ref{construction2} of the 
Erd\H{o}s construction.  As in Algorithm \ref{alg:grantham} there will be several running 
products; the goal is to get these products closer to the identity in $\Lunits$.
However, this time elements are matched for $Q_i$ with $i$ close to $r$ first, 
and at each level the first element matched is the distinguished product containing 
$b^{-1}$ modulo $\La$.
In this way the final identity product is $b^{-1}$ times a number of primes from a subset 
$\mathcal{T}$ of $\mathcal{P}$, 
making $\mathcal{P} \setminus \mathcal{T}$ the factors of a Carmichael number.

As discussed above, a heuristic application of the ideas from \cite{FlaxPrz05} results 
in a subset product algorithm that takes time and space $\Ot(2^{\sqrt{\log{\La}}})$.
However, this is still inefficient since it does not take advantage of the large number 
of $p \in \mathcal{P}$ with $\omega(p)$ small.  

Let $N_j$ be the size 
of the set $\mathcal{P}_j = \{ p \in \mathcal{P} \ : \ \omega(p) \leq j \}$.   
Then define a subgroup $G$ 
of $\Lunits$ as $(\Z/ \hat{\La} \Z)^{\times}$ with $\hat{\La} = \prod_{i=1}^m Q_i$, 
where 
$$
m = \min_{1 \leq j \leq r} j \mbox{ such that } 
\E[N_j] > (\log{\La}) 4^{\sqrt{ \sum_{i=1}^j h_i \log{q_i}}} \enspace .
$$
We will see in Section \ref{sec:bounds} that with reasonable assumptions 
the expected size of $N_j$ is at least $N \cdot \prod_{i=j+1}^r 1/(h_i+1)$.

For shorthand let $\ell = \sqrt{\log{|G|}}$.  
It is important that during the construction of $\mathcal{P}$ we pick out 
all elements of $\mathcal{P}_m$ and $\Theta (2^{\sqrt{\log{|G|}}} \log{|G|})$
elements of $\mathcal{P}$ with $\omega$ values equal to $j$ for $m < j \leq r$.
The elements with $\omega$ value greater than $m$ will be needed to match 
with $b^{-1}$, so that a product of primes and $b^{-1}$ has $\omega$ value $m$.
This product will be called the distinguished element $a_0$.
The elements of $\mathcal{P}_m$, along with $a_0$, are then matched 
over the course of $\ell$ levels.  At each level, products have another $\ell$ bits 
eliminated, so that by the end of $\ell$ levels an identity product has been formed.

Pseudocode is presented as Algorithm \ref{alg:shallue}.

\begin{algorithm} \label{alg:shallue}
\caption{Large Carmichael subset-product}
\tcc{Phase 1: find/construct elements of $G$}
{\bf Input:} set $S$ containing $\mathcal{P}_m$ and $(\log{|G|}) 2^{\sqrt{\log{|G|}}} $
primes $p$ with $\omega(p) = j$ for each of $m < j \leq r$ \;
$a_0 \leftarrow b^{-1} \bmod{\La}$ \tcc*[r]{$b$ is product of all elements of $\mathcal{P}$}
$\mathcal{T} = \emptyset$ \;
\For{$u \leftarrow r$ \KwTo $m$}{
     find $p \in \mathcal{P}_u$ such that $a_0 \cdot p \equiv 1 \bmod{Q_u}$ \;
     $a_0 \leftarrow a_0 \cdot p \bmod{\La}$, $\mathcal{T} \leftarrow \mathcal{T} \cup \{p\}$ \;
   }  
add $a_0$ to $S$ \;
\tcc{Phase 2: continually pair products to reach identity in $G$}
construct subgroups $G_i$, $1 \leq i \leq \ell$ with factor groups having size $2^{\ell}$ \;
\For{$i \leftarrow 1$ \KwTo $\ell$}{
  pair the element containing $a_0$ first to ensure it is included \;
  pair elements of $S$ whose product is in $G_i$ \;
}
\Return the history of any element in $G_{\ell} = \{1_G\}$ \;
\end{algorithm}

Constructing subgroups of the correct size is not too hard, since $|G|$ will typically be 
products of small primes to large powers.  As an example, note that if 
$G = (\Z/2^{h_1}\Z)^{\times}$ then $G'= \{ a \in G \ : \ a \equiv 1 \bmod{2^e} \}$ is a 
subgroup of $G$ of order $2^{h_1-e}$ and hence index $2^{e-1}$.  Thus for this 
example of $G$ we have $G_0 = G$ and 
$G_i = \{ a \equiv 1 \bmod{2^{i \lfloor \sqrt{|G|} \rfloor}} \}$.

\section{Symmetric Distributions}

Algorithms \ref{alg:grantham} and \ref{alg:shallue} are not guaranteed to succeed.
Rather, they will succeed with some positive probability depending upon the 
distribution of the elements of $\mathcal{P}$.  The key result proven in this 
section is that if the elements of $\mathcal{P}$ are distributed symmetrically, 
then the probability that the product of two elements is in a subgroup is at least 
as large as if the elements were distributed uniformly.

First, however, we mention the tail bound that will be used frequently in the analysis 
that follows.  Since products at any level are composed of distinct elements 
of $\mathcal{P}$, if the initial elements are independent then subsequent 
products are independent as well.  

\begin{theorem}[Chernoff bound] \label{thm:chernoff}
Let $X_1, X_2, \dots, X_n$ be independent Bernoulli trials that take 
value $1$ with probability $p_i$.  Let $X = \sum_{i=1}^n X_i$, $\mu = {\rm E}[X]$, 
and $\delta$ be any real number in the range $(0,1]$.  Then
$$
{\rm Pr}[ X < (1 - \delta) \mu] < {\rm exp}( - \mu \delta^2/2) \enspace .
$$
\end{theorem}

A classical result is that collision probability is minimized 
when the distribution is uniform.  For a proof see \cite[page 66]{Sinkov68}.

\begin{lemma} \label{lem:abstract_collision}
Let $X$ be a random variable on a set $S$.  Then 
$$
\sum_{a \in S} {\rm Pr}[X=a]^2 \geq \sum_{a \in S} \left( \frac{1}{|S|} \right)^2 \enspace .
$$
\end{lemma}

Theorems \ref{theorem:main1} and \ref{theorem:main2} assume that the given distribution 
is symmetric.  Our definition is a simple generalization of that in \cite{FlaxPrz05}.

\begin{definition}
The distribution of a random variable $X$ on $G$ is symmetric if 
${\rm Pr}[ X = a] = {\rm Pr}[ X = a^{-1} ] $ for all $a \in G$.
In this case we call $X$ a symmetric random variable.
\end{definition}

Recall that all groups under consideration are abelian.
If $H$ is a subgroup of $G$ and $X$ is a random variable on $G$
 then there is a natural random variable on $G/H$ given by sampling $X$ and then mapping 
 the result to $G/H$ via the map $g \mapsto gH$.  Call this random variable $X_H$.
 A specific example occurs when $G = \Lunits$, we sample $X$ and want the result 
 modulo $Q_i$.  Call this random variable $X \bmod{Q_i}$.
 Mapping to $G/H$ preserves the symmetric property.

\begin{proposition} \label{prop:sym_subgroup}
Let $H$ be a subgroup of $G$.  If $X$ is symmetric then $X_H$
is symmetric.
\end{proposition}
\begin{proof}
Let $\phi: G \rightarrow G/H$ be the canonical homomorphism.  Then
$$
\Pr[ X_H = \hat{a}H] = \sum_{a \in \hat{a}H} \Pr[X=a]
= \sum_{a^{-1} \in \hat{a}^{-1}H} \Pr[X = a]
= \sum_{a \in \hat{a}^{-1}H} \Pr[X = a^{-1}]
$$
since $\phi$ a homomorphism implies that 
$a \in \hat{a}H$ if and only if $a^{-1} \in \hat{a}^{-1}H$.  The fact that $X$ is symmetric
then yields
$$
\sum_{a \in \hat{a}^{-1}H} \Pr[X = a^{-1}]
= \sum_{a \in \hat{a}^{-1}H} \Pr[X = a] = \Pr[X_H = \hat{a}^{-1}H] \enspace .
$$
\end{proof}

Constructing a new group via direct product also preserves the symmetric property 
for random variables.  If $X_1$, $X_2$ are independent random 
variables on $H_1, H_2$ respectively, then let $X_1 \times X_2$ be a random variable 
on $G = H_1 \times H_2$.  We define this random variable by 
$$
{\rm Pr}[ X_1 \times X_2 = (a,b)] = {\rm Pr}[X_1 = a] \cdot {\rm Pr}[ X_2 = b] \enspace .
$$

\begin{proposition} \label{prop:sym_directprod}
If $X_1$, $X_2$ are symmetric random variables on $H_1, H_2$ respectively then 
$X_1 \times X_2$ is symmetric on $G = H_1 \times H_2$.
\end{proposition}
\begin{proof}
Follows immediately from the fact that $(a,b)^{-1} = (a^{-1}, b^{-1})$.

\end{proof}

Products of random variables also preserve the symmetric property.

\begin{proposition} \label{prop:sym_product}
Let $X_1, X_2, \dots, X_n$ be symmetric random variables on $G$.  Then 
$Y = \prod_{i=1}^n X_i$ is also a symmetric random variable on $G$.
\end{proposition}
\begin{proof}
Using the symmetric nature of each of the $X_i$ we have the following 
identity involving multiple sums.
\begin{align*}
& {\rm Pr}[ Y = b]  \\
&= \sum_{a_1, \dots a_{n-1} \in G} {\rm Pr}[X_1 = a_1]  \cdots 
			{\rm Pr}[X_{n-1} = a_{n-1}] {\rm Pr}[ X_n = b \cdot (a_1 a_2 \cdots a_{n-1})^{-1}] \\
&= \sum_{a_1, \dots a_{n-1} \in G} {\rm Pr}[X_1 = a_1^{-1}]  \cdots 
		{\rm Pr}[X_{n-1} = a_{n-1}^{-1}] {\rm Pr}[ X_n = b^{-1} \cdot (a_1 a_2 \cdots a_{n-1})] \\
&= {\rm Pr}[ Y = b^{-1}] \enspace .
\end{align*}
\end{proof}

Phase 2 of Algorithm \ref{alg:shallue} involves matching elements whose product 
is in some subgroup $H$ of $G$.
Weakening the definition of collision to having $X_1 \cdot X_2$ in a subgroup $H$ yields
a similar result to Lemma \ref{lem:abstract_collision}:  
symmetric random variables have a greater than uniform collision probability.

\begin{proposition} \label{prop:subgroup_collision}
Let $H$ be a subgroup of $G$ and let $X_1$, $X_2$ be independent random variables 
on $G$ with identical symmetric distributions.  Then the probability that $X_1 \cdot X_2$ is in 
$H$ is at least $|H|/|G|$.
\end{proposition}
\begin{proof}
Let $C$ be a set of coset representatives of $H$.  By group theory, $C$ has 
size $|G|/|H|$.  $X_i \bmod{H}$ is symmetric by Proposition \ref{prop:sym_subgroup}, 
and thus
$\Pr[ X_1 \in \hat{a}H] = \Pr[X_2 \in \hat{a}^{-1}H] $ for all $\hat{a}$ in $C$.
Then
\begin{align*}
\Pr[ X_1 \cdot X_2 \in H]
&= \sum_{a \in G} \Pr[ X_1 = a] \Pr[X_2 \in a^{-1}H] \\
&= \sum_{\hat{a} \in C} \sum_{a \in \hat{a}H} \Pr[ X_1 = a] \Pr[X_2 \in a^{-1}H] \\
&= \sum_{\hat{a} \in C} \Pr[ X_1 \in \hat{a}H] \Pr[X_2 \in \hat{a}^{-1}H] \\
&= \sum_{\hat{a} \in C} \Pr[ X_1 \in \hat{a}H]^2 \\
&\geq \sum_{\hat{a} \in C} \left( \frac{1}{|G|/|H|} \right) ^2 
= \frac{1}{|G|/|H|}
\end{align*}
where the lower bound follows from Lemma \ref{lem:abstract_collision}.
\end{proof}

\section{Divisors of $\La$} \label{sec:dist}

L\"{o}h and Niebuhr \cite{LohNie96} note the distribution of divisors of $\La$ modulo 
$q_i$, but provide no proof.  Here we give a full description of the distribution modulo
$Q_i = q_i^{h_i}$ in order to provide justification for the claim that elements of $\mathcal{P}$
are distributed symmetrically modulo $\La$.

We begin with the following lemma, then extend the distribution to all classes 
modulo $q_i^{h_i}$.

\begin{lemma} \label{lemma:div_dist}
Suppose that a divisor $d$ of $\La$ is chosen uniformly at random from the set of all 
divisors of $\La$.  Then
$$
\Pr[ d \mbox{ exactly divisible by } q_i^{e}] = \frac{1}{h_i+1}
$$
for all $1 \leq i \leq r$ and all $1 \leq e \leq h_i$.
\end{lemma}
\begin{proof}
The divisors of $\La$ can be identified with $r$-tuples $(e_1, \dots, e_r)$, where 
$e_i$ is power of $q_i$ that divides $d$.  It is thus the case that when the divisors are 
partitioned by the power of $q_i$, there are $h_i+1$ such partitions and they are 
all of the same size.
\end{proof}

If we now assume that the part of a divisor relatively prime to $q_i$ is uniformly distributed 
modulo $Q_i$, then the $1/(h_i+1)$ probability of all divisors exactly divisible by $q_i^e$
 is shared equally among the $q_i^{h_i-e}-q_i^{h_i-e-1}$ elements of $\Z/Q_i\Z$ which 
 are exactly divisible by $q_i^e$.  This yields the heuristic distribution 
 $$
{\rm Pr}[ d \equiv a \bmod{q_i^{h_i}} ] = \left\{
\begin{array}{cl} \frac{1}{h_i+1} & \mbox{if $a = 0$}  \vspace{1 mm} \\
			   \frac{1}{(h_i+1)(q_i^{h_i-e}-q_i^{h_i-e-1})} & \mbox{if $q_i^e$ exactly divides $a$}
\end{array} \right. 
$$ 
It is easy to show that if $d+1$, $d'+1$ are multiplicative 
inverses in $(\Z/Q_i\Z)^{\times}$ and $q_i^e$ exactly divides $d$ 
then $q_i^e$ exactly divides $d'+1$.
Thus the distribution of $X = d+1, d \mid \La$ is symmetric on $(\Z/Q_i\Z)^{\times}$, 
and Proposition \ref{prop:sym_directprod} extends this result to $\Lunits$.

A difficult question is whether the distribution remains symmetric with the added condition 
that $d+1$ be prime.  We will assume it does, but it is worth noting that elements 
of $\mathcal{P}$ do not have the same distribution as divisors of $\La$.  For example, 
if $d \equiv k q_i-1 \bmod{q_i^{h_i}}$ for some $k$ then $d+1$ is divisible by $q_i$ and thus 
not prime (note this particular problem does not jeopardize the claim that 
$\mathcal{P}$ is distributed symmetrically over $\Lunits$).

Analysis of Algorithm
\ref{alg:shallue} will depend upon the following heuristic assumptions, which we 
will henceforth call the standard assumptions.

\begin{heuristic}[Standard assumptions]
Let $X_p$ be a random variable corresponding to the value 
modulo $\La$ of $p \in \mathcal{P}$.
\begin{enumerate}
\item The $X_p$ are symmetric random variables on $\Lunits$.
\item The $X_p$ are independent, as are $X_p \bmod{Q_i}$ for different $1 \leq i \leq r$.
\item The probability that $X_p \equiv 1 \bmod{Q_i}$ is at least $1/(h_i+1)$.
\item $\La$ is constructed so that $1 \leq h_i \leq r/i$ for all $1 \leq i \leq r$.
\item The size of $\mathcal{P}$ is $K(\La)$.
\end{enumerate}
\end{heuristic}

\section{Algorithm analysis}

The following theorem provides the proof for Theorem \ref{theorem:main1}
and is general enough to be applicable to many settings besides
constructing Carmichael numbers.  
We use the notation 
$\ell$ for $\sqrt{\log_2{|G|}}$, and follow closely the proof of 
Theorem 3.2 from \cite{Kup05}.

\begin{theorem} \label{thm:fp_generalized}
Let $G = G_0$ be an abelian group with subgroups $G_1, G_2, \dots, G_{\ell}$ where 
$|G_i|/|G_{i+1}| = 2^{\ell}$ for $0 \leq i \leq \ell-1$.  Suppose that $S$ contains 
at least $O(\ell^2 4^{\ell})$ independent elements of $G$ distributed symmetrically.  
Then Phase 2 
of Algorithm \ref{alg:shallue} finds a solution to the subset product problem with 
probability at least $1 - e^{-\Omega(\log{|G|})}$ using time and space 
$\Ot(4^{\sqrt{\log_2{|G|}}})$.
\end{theorem}
\begin{proof}
We begin by assuming the algorithm has been successful up to level $u$, so 
that we have a list $L_u$ of elements in the group $G_u$.  Our goal is to prove 
by induction that 
$$
|L_u| \geq C_u \cdot \ell^2 2^{2\ell - u}
$$
with high probability, given that $|L_0| = C_0 \ell^2 2^{2\ell}$.  
Here $C_u$ is defined recursively by $C_0 = 3, 
C_{u+1} = C_{u} - 2^{-(\ell-u)}$.  For $0 \leq u \leq \ell$ we have $1 \leq C_u \leq 3$.
As long as $|L_{\ell}| \geq 1$, then Phase 2 succeeds in finding a solution.

Given $a \in L_u$, let $X_b$ be a Bernoulli random variable that takes value $1$ if 
$a$ and $b$ ``match," that is if $a \cdot b \in G_{u+1}$.  Then $a$ has a match in $L_u$
as long as $X = \sum_b X_b \geq 1$.  Elements of $G_u$ are  
products of symmetrically distributed elements of $S$, and thus are symmetrically 
distributed themselves by Proposition \ref{prop:sym_product}.  It then follows from 
Proposition \ref{prop:subgroup_collision} that $X_b = 1$ with probability at least 
$|G_u|/|G_{u+1}| = 2^{-\ell}$.  Thus $\E[X] \geq \ell^2$ until all but $\ell^2 2^{\ell}$ elements are 
matched.  If this holds then 
$$
|L_{u+1}| \geq \frac{ |L_u| -  \ell^2 2^{\ell} }{2}
\geq \ell^2 2^{2\ell - u - 1} \cdot \left( C_u -  2^{-(\ell-u)} \right) \enspace .
$$
The products are distinct and hence independent, so the Chernoff bound applies.
The probability that $a$ fails to match is at most
$$
\Pr[ X \leq \ell^2/2] \leq \Pr[ X \leq (1-1/2) \E[X]]
\leq {\rm exp}( - \E[X]/8) \leq {\rm exp}( - \ell^2/8) \enspace .
$$
We seek to make at most $\ell^2 4^{\ell}$ matches, so all will succeed with probability at least
$$
( 1 - e^{-\ell^2/8})^{\ell^2 4^{\ell}} \geq 1 - \ell^2 4^{\ell} e^{-\ell^2/8} \enspace .
$$
The probability of all matches succeeding at all $\ell$ levels is then at least
$$
(1 - e^{-\frac{\ell^2}{16} + 3 \ell})^{\ell}
\geq 1 - \ell e^{-\frac{\ell^2}{16} + 3\ell}
\geq 1 - e^{- \Omega(\ell^2)}
= 1 - e^{-\Omega(\log{|G|})} \enspace .
$$
\end{proof} 

\section{Bounding the Running Time} \label{sec:bounds}

The previous section gave a subexponential analysis of Phase 2 running on a general 
group $G$.  Here we bound $\log{|G|}$ in terms of $N$ so that the running time of 
Algorithm \ref{alg:shallue} can be expressed in terms of the problem size, proving 
Theorem \ref{theorem:main2}.  In fact, an easy bound on $\log{\La}$ would be 
sufficient asymptotically, but to measure the gain from having $G$ be a subgroup of 
$\Lunits$ we go further and prove that $m = O(\sqrt{r} \log{r})$.
Throughout we will take the standard assumptions 
as given.

Our starting 
point is item (5) of the standard assumptions, namely that 
$$
|\mathcal{P}| = 
\frac{\La}{\phi(\La) \ln{\sqrt{2\La}} } \prod_{i=1}^r \left( h_i + \frac{q_i-2}{q_i-1} \right) 
$$
and which we denote by $N$.
Recall that $G$ is defined as the integers modulo $\hat{\La} = \prod_{i=1}^m q_i^{h_i}$
where $m$ is the smallest integer such that $N_m$ is large enough.  Here $N_m$ 
is the number of elements of $\mathcal{P}$ with $\omega(p) \leq m$ and ``large enough"
means $\E[N_m]$ is large enough for Phase 2 to succeed with high probability.  By our standard assumptions the probability 
that an element of $\mathcal{P}$ is congruent to $1$ modulo $q_i^{h_i}$ is at least
$\frac{1}{h_i+1}$ and this condition is independent for different values of $i$.  
Hence the condition on $\E[N_m]$ becomes
$$
\frac{\La}{\phi(\La) \ln{\sqrt{2\La}} } \prod_{i=1}^m \left(h_i + \frac{q_i-2}{q_i-1} \right)
\prod_{i = m+1}^r \frac{h_i + \frac{q_i-2}{q_i-1} }{h_i+1}
>
(\log{\La}) 4^{\sqrt{\sum_{i=1}^m h_i \log{q_i}  }}  \enspace .
$$
Clearly $1 \leq m \leq r$ and a smaller $m$ is preferred so we will 
provide an upper bound.  The analysis requires $m > 10$ and $r \geq 64$.

Bounding $\log{|G|}$ requires several preparatory results.  First we prove that 
$\log{\La}$ is polynomial in $r$.  

\begin{lemma}\label{lemma:La_bound}
Assume that $10 < m \leq r $ and that $1 \leq h_i \leq  r/i $ for $1 \leq i \leq r$.  Then 
$$
m < \sum_{i=1}^m h_i \log{q_i} < 2 r (\log{m})^2  \enspace .
$$
\end{lemma}
\begin{proof}
The $q_i$ are 
the first $r$ primes and $1 \leq h_i$ for all $i$.  Thus 
$\sum_{i=1}^m h_i \log{q_i} > m$.

For the upper bound we require a bound on the $m$th prime number.  From 
\cite[Section 8.8]{Bach96} this is given by $q_m < m(\ln{m} + \ln{\ln{m}})$ for $m \geq 6$.
With $m > 6$ we use the conceptually easier bound of $\log{q_m} < 2\log{m}$.
  Since $h_i \leq  r/i $  this yields
$$
\sum_{i=1}^m h_i \log{q_i} < 2r\log{m} \sum_{i=1}^m\frac{1}{i}
< 2r\log{m} \cdot (1+\ln{m}) 
$$
and $(1 + \ln{m}) < \log{m}$ for $m > 10$.
\end{proof}


An immediate corollary is that $\log{\La} < 2r (\log{r})^2$.
Next we find that $r$ is logarithmic in $N$.

\begin{lemma}\label{lemma:N_bound}
Assume that $r \geq 64$.  Then $\log{N} > r/3$.
\end{lemma}
\begin{proof}
We start with the definition of $N$ given above.  The term $h_i + \frac{q_i-2}{q_i-1}$
is bounded below by $3/2$ for all $q_i \geq 3$.  Since $\La/\phi(\La) > 1$ and 
$\ln{\sqrt{2\La}} < \log{\La} < 2r(\log{r})^2$ by Lemma \ref{lemma:La_bound} we have 
$$
N = \frac{\La}{\phi(\La) \ln{\sqrt{2\La}} } \prod_{i=1}^r \left( h_i + \frac{q_i-2}{q_i-1} \right) 
> \frac{1}{2r (\log{r})^2  } \left( \frac{3}{2} \right)^{r-1}
$$
Thus $\log{N} > (r-1)\log{(3/2)} - (1 + \log{r} + 2\log{\log{r}}  ) > r/3$ when $r \geq 64$.
\end{proof}

When bounding $\E[N_m]$ a sticky term is $\prod (h_i+1) / (h_i + \frac{q_i-2}{q_i-1})$.
It ought to be close 
to one;  we provide a bound logarithmic in $r$.

\begin{lemma}\label{lemma:correctionterm_bound}
Assume $r \geq 16$.  Then for all $m \geq 1$
$$
\prod_{j=m+1}^r \frac{h_j+1}{h_j+\frac{q_j-2}{q_j-1}} < (\ln{r})^2 \enspace .
$$
\end{lemma}
\begin{proof}
Start with the transformation
$$
\frac{h_j+1}{h_j+\frac{q_j-2}{q_j-1}} = 
\frac{h_j + \frac{q_j-2}{q_j-1} + 1 - \frac{q_j-2}{q_j-1} }{ h_j+\frac{q_j-2}{q_j-1}}
= 1 + \frac{1/(q_j-1)}{h_j+\frac{q_j-2}{q_j-1}} 
= 1 + \frac{1}{(q_j-1)h_j + q_j-2} \enspace .
$$
By assumption $q_j \geq 3$ and $h_j \geq 1$, so 
$(q_j-1)h_j \geq 2$ for all $2 \leq j \leq r$, giving the bound
$$
1 + \frac{1}{(q_j-1)h_j + q_j-2} \leq 1 + \frac{1}{q_j} \enspace .
$$
We saw in Lemma \ref{lemma:La_bound}
that $r \geq 6$ implies $q_r < 2r\ln{r}$.  We now use a result 
from \cite[Section 8.8]{Bach96} on the prime reciprocal sum, namely
$$
\sum_{j=1}^r \frac{1}{q_j} \leq \sum_{q < 2r\ln{r}} \frac{1}{q}
< \ln{\ln{ (2r\ln{r})}} + B + \frac{1}{(\ln{(2r\ln{r})})^2} \enspace .
$$
where the sums are over primes 
and $B < 0.27$ is the prime-reciprocal constant.  This upper bound 
is less than $2\ln{\ln{r}}$ as long as $r \geq 16$.
The slope of the tangent line to $y=e^x$ at $x=0$ is $1$, which means $e^x \geq 1+x$
for all positive $x$.  With $x = 1/q_j$ this gives
$$
\prod_{j=m+1}^r \left( 1 + \frac{1}{q_j} \right)
< {\rm exp} \left( \sum_{j=1}^r \frac{1}{q_j} \right)
< {\rm exp} (2\ln{\ln{r}})
= (\ln{r})^2 \enspace .
$$
\end{proof}

With this preparatory work out of the way the bound on $\log{|G|}$ follows from 
properly bounding ${\rm E}[N_m]$, the expected number of elements of $\mathcal{P}$
with $\omega$-value $m$.

\begin{lemma}\label{lemma:G_bound}
Assume $r$ and $N$ are sufficiently large.  Given the standard assumptions, 
$$
\log{|G|} < 27 \log{N} (\log{\log{N}})^2 \enspace .
$$
\end{lemma}
\begin{proof}
Recall that $G$ is defined as the group of units modulo $\prod_{i=1}^m q_i^{h_i}$ 
where $m$ is the smallest integer such that 
$N \prod_{i=m+1}^r \frac{1}{h_i+1} > (\log{\La}) 4^{\sqrt{\log{|G|}}}$.  Since $m$ is the smallest 
such integer, multiplying by $1/(h_m+1)$ flips the inequality.  Thus
$$
\frac{1}{h_m+1} 
\frac{\La}{\phi(\La) \ln{\sqrt{2\La}} } \prod_{i=1}^m \left(h_i + \frac{q_i-2}{q_i-1} \right)
\prod_{i = m+1}^r \frac{h_i + \frac{q_i-2}{q_i-1} }{h_i+1}
<
(\log{\La}) 4^{\sqrt{\log{|G|}  }} \enspace .
$$
Focusing on the left hand side, $h_m < r$ by construction, 
$\La/\phi(\La) > 1$, $\ln{\sqrt{2\La}} < \log{\La} <2r (\log{r})^2$ by Lemma \ref{lemma:La_bound}, 
and $\prod (h_i + \frac{q_i-2}{q_i-1})(h_i+1) > (\ln{r})^{-2}$ 
by Lemma \ref{lemma:correctionterm_bound}.  Combining all these bounds yields
\begin{align*}
&  \frac{1}{r+1} \frac{1}{2r(\log{r})^2} \frac{1}{(\ln{r})^2}
\prod_{i=1}^m \left(h_i + \frac{q_i-2}{q_i-1} \right)
<
(\log{\La}) 4^{\sqrt{\log{|G|} }} \\
\implies &
\sum_{i=1}^m \log{ \left( h_i + \frac{q_i-2}{q_i-1} \right)} - 
				(2\log{2r} + 4\log \log{r})
< \log \log{\La} + 2 \sqrt{ \sum_{i=1}^m h_i \log{q_i} } \\
\implies &
m \log{(3/2)} - 3\log{r} < \log\log{\La} + 2\sqrt{2 r (\log{m})^2}
\end{align*}
where  $\sum_{i=1}^m h_i \log{q_i} < 2r (\log{m})^2$ follows from 
Lemma \ref{lemma:La_bound}.

If $m > 6 \sqrt{r} \log{r}$ then $m\log{(3/2)} > 2\sqrt{2r} \log{m} + \log\log{\La} + 3\log{r}$ 
for sufficiently large $r$ so we must 
have $m < 6 \sqrt{r} \log{r}$.

Since $\log{|G|} < 2r (\log{m})^2$ this bound on $m$ gives us 
$\log{|G|} < 2r (\log{6} + \frac{1}{2} \log{r} + \log\log{r})^2$
which is at most $2r (\frac{3}{2} \log{r})^2$ if $r \geq 32$.
Lemma \ref{lemma:N_bound} now completes the proof.

\end{proof}

Although asymptotically $\log{|G|}$ and $\log{\La}$ are equivalent at 
$O(\log{N} (\log\log{N})^2)$, this work showing $m = O(\sqrt{r}\log{r})$
provides a theoretical justification for the gains seen in practice.  
We now prove 
Theorem \ref{theorem:main2}, restated here for convenience.  It was proven 
that $m < 6\sqrt{r}\log{r}$ in Lemma \ref{lemma:G_bound} 
so the necessary assumption that  $m > 2 \sqrt{r}\log{r}$ causes no harm.

\begin{theorem} 
Let $G = \Lunits$ and $\mathcal{P}$ be defined as in the Erd\H{o}s construction, with the 
added assumption that $1 \leq h_i \leq r/i$ for all $1 \leq i \leq r$.
Assume that the elements of $\mathcal{P}$ are independent and 
distributed symmetrically in $G$, and 
that $N = | \mathcal{P} | = K(\La)$.  Finally, assume that the probability $p \equiv 1 
\bmod{Q_i}$ for $p \in \mathcal{P}$ is at least $1/(h_i+1)$ and independent across $Q_i$.

Then there is an algorithm that with probability at least $1 - e^{-\Omega(\log{\La})}$
finds a subset of $\mathcal{P}$ that products to $b$ in $G$ and requires time and space
$$
2^{ O( \sqrt{  (\log{N}) (\log\log{N})^2 } ) } \enspace .
$$
\end{theorem}
\begin{proof}
First assume a solution is found and that it includes $a_0$, the distinguished element of $G$.
Then $b^{-1} \bmod{\La}$ times some product of primes is the identity in $G$, and since each 
$a_i$ is the identity in $\Lunits / G$, we have discovered a set of primes in $\mathcal{P}$ that 
product to $b$ modulo $\La$.

In bounding the probability of failure we focus first on the size of $\mathcal{P}_m$
and the distinguished element $a_0$.  We notated $|\mathcal{P}_m|$ by $N_m$, 
and chose $G$ so that 
$$
\E[N_m] > (\log{\La}) 4^{\sqrt{\log{|G|}}} \enspace .
$$
For $p \in \mathcal{P}$ let $X_p$ be a Bernouilli random variable that takes value 
$1$ if $\omega(p)=m$, and let $X$ be the sum of all $X_p$.  Then $\E[X] = \E[N_m]$
and by the Chernoff bound
\begin{align*}
& {\rm Pr}[ X < (1-1/2) \E[N_m]] < {\rm exp}(- \E[N_m]/8) \\
\implies & {\rm Pr}\left[X < \frac{1}{2} (\log{\La}) 4^{\sqrt{\log{|G|}}} \right] 
	< {\rm exp}\left( -\frac{\log{\La}}{8}  \right)
\enspace .
\end{align*}
The distinguished element is dealt with in a different fashion.  For levels $u=r$ to $m+1$
we multiply $b$ by an element $p$ such that $\omega(p)=u$ and $p \equiv b \bmod{q_u^{h_u}}$.
Focus on level $u$ and let $Y_p$ be a Bernoulli random variable that is $1$ if it satisfies that 
condition, with $Y$ being the sum over all $Y_p$.  The worst case is at level $m+1$, where 
$$ 
{\rm Pr}[ Y_p = 1] \geq \frac{1}{q_{m+1}^{h_{m+1}} - q_{m+1}^{h_{m+1}-1} }
\prod_{i=m+2}^r \frac{1}{h_i+1}
\geq \frac{1}{q_{m+1}^{h_{m+1}}} \prod_{i=m+2}^r \frac{1}{h_i+1}
$$
and hence
$$
\E[Y] \geq \frac{N}{q_{m+1}^{h_{m+1}}} \prod_{i=m+2}^r \frac{1}{h_i+1}  
\geq \frac{1}{q_{m+1}^{h_m+1}} \cdot \E[N_m] 
\geq \frac{4^{\sqrt{\log{|G|}}}}{q_{m+1}^{h_m+1}}  \cdot \log{\La} \enspace .
$$
We will show $(h_{m+1} \log{q_{m+1}})^2 < 4 \sum_{i=1}^m h_i\log{q_i}$
and hence that $q_{m+1}^{h_{m+1}} < 4^{\sqrt{\log{|G|}}}$.
Since the $h_i$ are non-increasing, $4 \sum_{i=1}^m h_i \log{q_i} \geq 4 m h_{m+1}$.
Meanwhile, $$h_{m+1} (\log{q_{m+1}})^2 \leq \frac{r}{m+1} \cdot ( 2\log{(m+1)})^2 \leq 4m$$
since $m > 2 \sqrt{r} \log{r}$ implies $m^2 + m \geq 4 r (\log{r})^2$.

Now the Chernoff bound can again be applied, giving
$$
{\rm Pr}[Y < (1 - 1/2)(\log{\La})] \leq {\rm exp}\left( - \log{\La}/8 \right) \enspace .
$$
This is the worst case among the at most $r$ levels, which means the total success 
probability of Phase 1 is at least 
$$
(1 - e^{-\Omega(\log{\La})})(1 - e^{-\Omega(\log{\La})})^r
\geq 1 - (r+1)e^{-\Omega(\log{\La})}
> 1 - (3\log{N}+1)e^{-\Omega(\log{\La})}
$$
where $r < 3\log{N}$ follows by Lemma \ref{lemma:N_bound}.

With $S$ containing $O((\log{\La}) 4^{\sqrt{\log{|G|}}})$ elements, Theorem
\ref{thm:fp_generalized} allows us to conclude that Phase 2 completes with 
probability at least $1 - e^{-\Omega(\log{\La})}$.
Note that the distinguished element is 
distributed symmetrically by Proposition \ref{prop:sym_product} and that Theorem 
\ref{thm:fp_generalized} assumes it is always the first to be matched.
The total probability of success is thus also $1 - e^{-\Omega(\log{\La})}$.

Finally we note the running time.  We start the algorithm with at most \newline
$O(r (\log{\La}) 4^{\sqrt{\log{|G|}}})$ elements in $S$.  
Finding the right product to put $b^{-1}$ in $G$  takes 
$r$ searches at a cost of $O(\sqrt{\log{|G|}})$ each.  Phase 2  takes time 
and space $\Ot( (\log{\La}) 4^{\sqrt{\log{|G|}}})$.  Applying Lemma \ref{lemma:G_bound}
then gives a total resource usage of $2^{ O( \sqrt{  (\log{N}) (\log\log{N})^2 } ) }$.
\end{proof}

\section{Construction} \label{sec:implementation}

Algorithm \ref{alg:grantham} was implemented in Python and run on a T3E.  
The best result so far 
is that Carmichael numbers have been constructed with $k$ prime factors 
for every $k$ between $3$ and $19565220$.

It is important to detail how one can efficiently ``push down" several 
primes in one step.  Let $a$ be an element of the set $S$ at stage $j$.  We 
seek $\hat{a}$ that maximizes $\bar{\omega}(a \cdot \hat{a})$.
For a set $A$ let $\bar{A} = \{a^{-1} \bmod{\La} : \ a \in A\}$.
Then combine the sets $S'=\{a \ : \ \mbox{$a\in S$ and $a<a^{-1} \bmod{\La}$}\}$
and $\overline{S \setminus S'}$ and sort lexicographically on the list of residues.
Then for our element $a$, the element of $\overline{S \setminus S'}$ closest 
to $a$ in the $\bar{\omega}$ metric will be the element immediately following 
$a$ or preceding $a$.  We use $S'$ rather than combining $S$ and $\bar{S}$
because in the latter case we end up with duplicates.
Once you
find the associated $\hat{a}$, if the maximum possible value of $\bar{\omega}(a \cdot \hat{a})$
at the $j$th stage is $j$, proceed to the next element. If it is
greater than $j$, then remove both $a$ and $b$ from $S$, multiply and
set aside. 

Algorithm  \ref{alg:shallue} was implemented using C++
and NTL \cite{NTL} and run on a 3.3 GHz  processor
with 16 GB of main memory.
Two Carmichael numbers are presented in Table \ref{table:construction}.
Here, $k$ is the number of prime factors of $n$
and $d$ denotes the number of decimal digits of $n$, 
while $\mathcal{T}$ is the set output by Algorithm \ref{alg:shallue}
and removed.  As conjectured in 
\cite{LohNie96}, $K(\La)$ is within 3\%  of
$|\mathcal{P}|$. 
The last thirty digits of each Carmichael number are also included.
The billion prime case took more time because the method of 
calculating $G$ resulted in Phase 2 operating on 27.7 million primes, 
as opposed to 16.5 million for the ten billion factor case.  This neatly demonstrates
how the number of primes needed for Phase 2 grows more slowly than $N$.

{\small
\begin{table}
\caption{Large Carmichael numbers} \label{table:construction} 
{\small
\begin{tabular}{l}
\hline
\noalign{\smallskip}
$\Lambda = 2^{15} \! \cdot \! 3^8 \! \cdot \! 5^5 \! \cdot \! 7^4 \!
\cdot \! 11^3 \! \cdot \! 13^2 \! \cdot \! 17^2 \! \cdot \! 19^2 \!
\cdot \! 23^2 \! \cdot \! 29 \! \cdot \! 31 \! \cdot \! 37  \! \cdot
\! 41 \! \cdot \! 43 \! \cdot \! 47 \! \cdot \! 53 \! \cdot \! 59 \!
\cdot \! 61 \! \cdot \! 67 \! \cdot \! 71 \! \cdot \!
73 \! \cdot \! 79$ \\
$=288828494392627542423975683172283292832395366400000$ \\
$k = 1021449117$, $|\mathcal{P}|=1021449926, K (\Lambda)$ = 1009441849  \\
$n = \ldots 547538202025813003668377600001$, 
$|\mathcal{T}| = 809$, $d= 25564327388$ \\
Subset Product Time = 174 sec \\
\hline
\noalign{\smallskip}
$\Lambda = 2^{16} \! \cdot \! 3^7 \! \cdot \! 5^5 \! \cdot \! 7^4 \!
\cdot \! 11^3 \! \cdot \! 13^2 \! \cdot \! 17^2 \! \cdot \! 19^2 \!
\cdot \! 23^2 \! \cdot \! 29^2 \! \cdot \! 31 \! \cdot \! 37  \! \cdot
\! 41 \! \cdot \! 43 \! \cdot \! 47 \! \cdot \! 53 \! \cdot \! 59 \!
\cdot \! 61 \! \cdot \! 67 \! \cdot \! 71 \! \cdot \!
73 \! \cdot \! 79 \! \cdot \! 83 \! \cdot \! 89 \! \cdot \! 97$ \\
$=4001166357176246301338040166195304168348080373267865600000$ \\
$k =10333229505, |\mathcal{P}| =10333230324, K (\Lambda) = 10225023621$\\
$n = \ldots  445706495205032238479360000001$, 
$|\mathcal{T}| = 819$,
$d = 295486761787$ \\
Subset Product Time = 98 sec \\
\hline
\end{tabular}
}
\end{table}
}

For Algorithm \ref{alg:shallue} the most important implementation
detail was the
instantiation of elements of $\mathcal{P}$.  Since such primes have the property
that $p-1$ divides $\Lambda$, one can store only the exponents of the divisors
of $\Lambda$, thereby fitting primes into a single 64-bit $\verb=long=$ for all
cases under consideration.  


However, one cannot multiply elements of $\mathcal{P}$ and maintain the property
that one less than the element divides $\Lambda$.  We created a class
$\verb=ModElement=$
which encapsulates an integer modulo $\Lambda$, a vector of condensed elements
of $\mathcal{P}$ that product to the integer (called the ``history"), 
and methods that product
such elements
or compute its $\omega$ value.  In this way, when some element is the
identity, the
history of that element gives the solution to the subset product problem.

We implemented the subgroups $G_i$ as integers modulo $M_i$ where $M_i$ is an
appropriate divisor of $\Lambda$.  Ease of implementation made this an
attractive
alternative to a generic group model, but one disadvantage is that
$M_i$ is sometimes
larger than the order of the subgroup it represents.  For example, if
$5^2$ divides $\Lambda$
and $5$ divides $M_i$, then $M_{i+1}$ will be divisible by $5^2$ in order that products
at the next level be congruent to $1 \bmod{25}$, and hence be five
times too big.

A significant improvement to Algorithm \ref{alg:shallue} comes from
passing elements
congruent to $1$ modulo $M_i$ to the next level without matching.
Since elements
congruent to $1$ modulo $q_i^e$ are more common, observed sizes of $L_i$ are
larger than $L_{i-1}/2$.

As for the underlying data structure, the necessary requirements are
that one have
constant time insertion and removals, and that one can quickly find
elements with
a particular value modulo $M_i$.  Our solution was a hash table of $\verb=ModElements=$
keyed to 
the residues modulo $M_i$ of the value.

Even though this data structure makes each prime expensive to store, the total
memory needed is quite manageable due to the relatively small number of primes
needed out of the total.  For example, in the ten billion case we use
only about
16.5 million primes, since $G = (\Z/ \hat{\La} \Z)^{\times}$ where
$\hat{\La} = 2^{16} \! \cdot \! 3^7 \! \cdot \! 5^5 \! \cdot \! 7^4 \!
\cdot \! 11^3 \! \cdot \! 13^2 \! \cdot \! 17^2 \! \cdot \! 19^2 \!
\cdot \! 23^2 \! \cdot \! 29^2 \! \cdot \! 31 \! \cdot \! 37  \! \cdot
\! 41 \! \cdot \! 43 \! \cdot \! 47 \! \cdot \! 53 \! \cdot \! 59$.

\section{Future Work}

In \cite{GuiMor92} the authors extend the basic Erd\H{o}s construction to a variety 
of other pseudoprimes.  Most likely the methods in this paper can be extended as 
well.

\bibliographystyle{amsplain}

\providecommand{\bysame}{\leavevmode\hbox to3em{\hrulefill}\thinspace}
\providecommand{\MR}{\relax\ifhmode\unskip\space\fi MR }
\providecommand{\MRhref}[2]{%
  \href{http://www.ams.org/mathscinet-getitem?mr=#1}{#2}
}
\providecommand{\href}[2]{#2}

\end{document}